\documentclass[11pt, oneside]{amsart}   	
\usepackage{geometry,graphicx,amssymb,amsmath,latexsym,amsthm}                		
\newtheorem{thm}{Theorem}[section]
\newtheorem{lem}[thm]{Lemma}
\newtheorem{prop}[thm]{Proposition}

\newcommand{\R}{\mathbb{R}}
\newcommand{\N}{\mathbb{N}}
\newcommand{\Z}{\mathbb{Z}}

\title{Transitivity of Infinite-Dimensional Extensions of Anosov Diffeomorphisms}
\author{Zev Rosengarten$^1$ and Asaf Reich$^2$}
\date{August 2012}						

\begin{document}
\begin{abstract}
We consider extensions of Anosov diffeomorphisms of an infranilmanifold by the real vector space $\R^\omega$. Our main result, based on the analogous theorem in finite dimensions proven by Nitica and Pollicott, is that any H$\ddot{o}$lder cocycle satisfying an obvious obstruction induces a topologically transitive extension (topologically weak mixing, in fact). We show how to construct cocycles satisfying these conditions for any Anosov diffeomorphism, and then observe that unlike the finite dimensional case, where cocycles satisfying the obstruction are $C^0$-stably transitive, there can be no infinite-dimensional stably transitive cocycles, with respect to several spaces and metrics of cocycles.
\end{abstract}

\maketitle
\let\thefootnote\relax\footnotetext{$^1$zrosenga@princeton.edu}
\let\thefootnote\relax\footnotetext{$^2$azmreich@math.brown.edu}
\let\thefootnote\relax\footnotetext{This research was performed by the authors at the 2012 Penn State University Math REU, funded by NSF grant 428-44 64BB. We would like to thank our mentors in this project, Viorel Nitica and Misha Guysinsky, who made this work possible. We would also like to thank the National Science Foundation and the Penn State math department.}
\section{Introduction}

In this paper we study the transitivity of certain dynamical systems of interest.  The definitions of transitivity differ slightly among authors (though they will agree on ``nice" spaces), so we will specify that a map $T: X \to X$ ($X$ a topological space) is called transitive if there exists $x_0 \in X$ such that the forward orbit $\{T^n(x_0) | n \in \N \}$ is dense in $X$. This natural condition can be thought of as some kind of ``chaotic" behavior by the system. A certain kind of maps which have proven to be particularly interesting and amenable to analysis are called Anosov. They are diffeomorphisms of a manifold such that the tangent bundle can be written as a sum $TX=E^s \oplus E^u$ in a continuous way, and with respect to some Riemannian metric the lengths of vectors in $E^s$ shrink exponentially under $T$ by some factor $\lambda$, while growing exponentially by $\lambda$ if they are in $E^u$. (In all of this paper, we will sometimes use a Riemannian metric on $X$ for definitions and results, but they will never depend on the choice of the metric, only the topology of $X$.) Anosov maps are thus ``hyperbolic": at each point there are some directions in which $T$ expands distances, and others in which it contracts them, but no directions that have distance preserved. The prototypical examples are matrices with integer entries in $SL(n,\Z)$ with no eigenvalues of norm 1, acting on the $n$-torus $\mathbb{T}^n =\R^n/\Z^n$ (for instance $(\begin{smallmatrix}2 &1\\ 1 &1\end{smallmatrix})$). Anosov diffeomorphisms have many nice properties and there are powerful tools for understanding their behavior, for example the Anosov closing lemma we use later states that if $T$ is Anosov, points that are ``almost" $n$-periodic are approximated by actual $n$-periodic points.  A major conjecture in smooth dynamics is that all Anosov maps are transitive. All known examples of such maps are proven to be transitive - in fact they take place on infranilmanifolds, which is conjectured to be necessary for the existence of an Anosov map, and transitivity has been proven in the infranilmanifold case (we will not go into the precise definition of such manifolds.)

We will study \emph{extensions} of such maps. A map $T: X \to X$ has an extension $T': Y \to Y$, if $Y$ is a space with surjective continuous map $\pi: Y \to X$ such that $\pi \circ T'=T \circ \pi$. An example is a map of a manifold being extended, by its differential, to the tangent bundle. This paper is one of many that have studied transitivity of extensions of hyperbolic maps; while we generalize \cite{NP}, other papers include \cite{N} which uses a shift for the base map, and \cite{MNT} which studies extension by groups similar to the special Euclidean groups.

\section{Statement of Result}

Let $X$ be a Riemannian manifold, $T: X \rightarrow X$ a transitive Anosov diffeomorphism. Given 
a group $G$ and a map $f: X \rightarrow G$ called a cocycle, we consider the $G$-extension $T_f : X\times G 
\rightarrow X \times G, T_f(x, g) : = (T(x), f(x)\cdot g)$. We say 
that $f$ is a transitive cocycle if $T_f$ is transitive. We say two cocycles $f$, $g$ are cohomologous 
if there exists a function $V: X \rightarrow G$ such that $f(x) = V(T(x))g(x)V(x)^{-1}$ for all $x \in X$. In 
this paper we will study the transitivity of $G$-extensions in the case that $G$ is an infinite-dimensional real vector space.

To this end, consider the space $\R^\omega = \prod _{n = 1}^{\infty} \mathbb{R}$ with the product topology (equivalent to the topology of pointwise convergence of sequences). This is a complete metric space with the metric

\[ 
d({a_n}, {b_n}) := \sum_{n = 1}^{\infty} {2^{-n}\frac{|a_n - b_n|}{1 + |a_n - b_n|}}.
\]

Our goal is to study transitivity of $\mathbb{R}^\omega$-extensions (which we'll now write additively) of Anosov diffeomorphisms for 
which the cocycle is H$\ddot{o}$lder continuous. So, as before, let $T: X\rightarrow X$ be an Anosov 
diffeomorphism, $f: X\rightarrow \mathbb{R}^\omega$ a H$\ddot{o}$lder cocycle. An obvious 
obstruction to transitivity of $f$ is that $f(X)$ be contained on one side of a hyperplane; that is, there exists a 
continuous nonzero linear functional $L: \mathbb{R}^\omega \rightarrow \mathbb{R}$ such that 
$L(f(x)) \geq 0$ for all $x \in X$. A more general obstruction to transitivity is when $f$ is cohomologous 
to such a cocycle. Let

\begin{center} $P_f := \{ \sum_{k = 0}^{m - 1} {f^k(x)}$, where $T^m(x) = x \}$ \end{center}
be the set of periodic data for $f$. By the following result, $f$ being cohomologous to a cocycle whose image is contained in $\{y | L(y) \geq 0\}$ is equivalent to having $P_f$ be contained in this set. For the proof, see \cite{Bou}.

\begin{thm}[Positive Livsic Theorem] Let $X$ be a Riemannian manifold, $T: X\rightarrow X$ a transitive Anosov diffeomorphism, and $f: X \rightarrow \mathbb{R}$ a H$\ddot{o}$lder cocycle. If $P_f \subset \mathbb{R}^+ = \{x \in \mathbb{R} | x \geq 0\}$, then $f$ is cohomologous to a H$\ddot{o}$lder cocycle with image contained in $\mathbb{R}^+$. \end{thm}
Thus, a necessary condition for $f$ to be transitive is that $P_f$ not be contained in any hyperplane, for if $L(p) \geq 0$ for all $p\in P_f$, where $L$ is a continuous linear functional, then $L\circ f$ is cohomologous to a function with values in $\mathbb{R}^+$, hence is not a transitive cocycle from $X$ to $\mathbb{R}$, so $f$ cannot possibly be transitive. We call a cocycle $f$ for which $P_f$ is not contained in any hyperplane inseparable. 

At this point, though we will not need it, it will be instructive to describe the continuous linear functionals on $\mathbb{R}^\omega$. The following result says that each such functional is essentially a functional on $\mathbb{R}^n$ for some $n$.

\begin{prop} Let $L:\mathbb{R}^\omega \rightarrow \mathbb{R}$ be a continuous linear functional. Then for some $n$ and some linear functional $A: \mathbb{R}^n \rightarrow \mathbb{R}$, $L = A\circ \pi_n$. \end{prop}
\begin{proof} Since $L$ is continuous, there exists $\epsilon > 0$ such that for any $x\in \mathbb{R}^\omega$ with $d(x, 0) < \epsilon$, we have $|L(x)| < 1$. Then if we choose $n$ such that $\sum_{i = n + 1}^\infty {2^{-i}} < \epsilon$, then for any $x\in \mathbb{R}^\omega$, $|L(x) - L(\pi_n(x))| < 1$. But if $L'$ is a nonzero linear functional on any real vector space $V$, then $L'(V) = \mathbb{R}$, since $L'(\alpha v) = \alpha L'(v)$ for all $\alpha \in \mathbb{R}$, and $v\in V$. Thus, $L - L\circ \pi_n$ must be the zero functional, i.e. $L = L_{|\mathbb{R}^n}\circ \pi_n$, as desired. \end{proof}

Our main result is that inseparability is not only necessary but also sufficient for transitivity when $X$ is an infranilmanifold.
\begin{thm} \label{main} Let $X$ be an infranilmanifold, $T: X \rightarrow X$ an Anosov diffeomorphism, and $f: X \rightarrow \mathbb{R}^\omega$ a H$\ddot{o}$lder cocycle. Then $f$ is transitive if and only if it is inseparable. \end{thm}
Our proof relies upon the following result from \cite{NP}.

\begin{thm} \label{npthm} Let $X$ be an infranilmanifold, $T: X \rightarrow X$ an Anosov diffeomorphism, and $f: X \rightarrow \mathbb{R}^n$ a H$\ddot{o}$lder cocycle. Then $f$ is transitive if and only if it is inseparable. \end{thm}

\section{Proof of Main Result}

\begin{proof}[Proof of Theorem \ref{main}] Let $X$, $T$, $f$ be as in the statement of the theorem. Consider 
for each $n \geq 1$ the natural subspace $\mathbb{R}^n  = \{\{x_i\} | x_i = 0$ for $i > n\}$ of $
\mathbb{R}^\omega$. Let $\pi_n : \mathbb{R}^\omega \rightarrow \mathbb{R}^n$ be the projection, 
which just returns the first $n$ coordinates. Then we have the following \begin{lem}
$f_n := \pi_n \circ f$ is H$\ddot{o}
$lder with respect to the usual Euclidean metric on $\mathbb{R}^n$. \end{lem}
\begin{proof}[Proof of lemma.] Since $X$ is compact, 
the projection of $f(X)$ onto each component is bounded. In particular, for some constant $C > 0$, $ 
|a_i - b_i| \leq C$ for all $\{a_i\}$, $\{b_i\} \in f(X)$ and all $i \leq n$. Thus, if $d(f(x), f(y)) \leq C'd_X(x, 
y)^\alpha$ for all $x$, $y\in X$, then for each $i \leq n$, and $x$, $y\in X$,
\[
|f(x)_i - f(y)_i| \leq \left(2^{-i}\frac{|x_i - y_i|}{1 + |x_i - y_i|}\right)2^n(1 + C),
\]
so
\[
\sum_{i = 1}^n {|f(x)_i - f(y)_i|} \leq 2^n(1 + C)d(f_n(x), f_n(y)) \leq C''d_X(x, y)^\alpha,
\]
so $f_n$ is H$\ddot{o}$lder with respect to the usual Euclidean metric on $\mathbb{R}^n$.
\end{proof}
In particular, we may apply Theorem 1.3 to $f_n$. Now $f_n: X\rightarrow \mathbb{R}^n$ is inseparable for each $n$, because $f$ is inseparable. Therefore, by Theorem 1.3, it is transitive (as a cocycle to $\mathbb{R}^n$). Now I claim that there exists $x \in X$ such that $(x, 0)\in X \times \mathbb{R}^n$ is a transitive point for $T_{f_n}$ for each $n$. Indeed, for each $n$ let $\{U_{nm}\}_{m\in \mathbb{Z}}$ be a countable basis for $X\times \mathbb{R}^n$. Then we have
\begin{equation} \{x\in X | (x, 0) \text{ is a transitive point for each } T_{f_n}\} = \bigcap_{n\in \mathbb{N}} \bigcap_{m\in \mathbb{Z}} \bigcup_{k\in \mathbb{N}} B_{knm} \end{equation}
where $B_{knm} := \{x\in X | T_{f_n}^k(x, 0)\in U_{nm} \}$. Now clearly $\cup_{k\in \mathbb{N}} B_{knm}$ is open, and we claim that it is dense in $X$ for each $n$ and $m$. That $\{x\in X | (x, 0)$ is a transitive point for each $T_{f_n}\} \neq \phi$ then follows from the Baire Category Theorem, since $X$ is a complete metric space. So suppose $x\in X$. Then since each $f_n$ is transitive, arbitrarily close to $x$ we may find $x'\in X$ such that $(x', v)$ is transitive for $T_{f_n}$ for some $v\in \mathbb{R}^n$. But it's clear from the way that an extension is defined that transitivity of a point doesn't depend on the second component, so $(x', 0)$ is transitive for $T_{f_n}$, and in particular $(x', 0)\in \cup_{k\in \mathbb{N}} B_{knm}$. Since we could choose $x'$ arbitrarily close to $x$, $\cup_{k\in \mathbb{N}} B_{knm}$ is dense, as claimed. Thus, we have a point $x\in X$ such that $(x, 0)$ is transitive for each $T_{f_n}$. I claim that $(x, 0)$ is a transitive point for $T_f$. Indeed, suppose $\epsilon > 0$ and $(y, v) \in X\times \mathbb{R}^\omega$. We must show that for some $k\geq 0$, $d_X(f^k(x), y) < \epsilon$ and $d(\beta_k(x), v) < \epsilon$, where $\beta_k(x) := \pi_{\mathbb{R}^\omega}(T_f^k(x, 0))$. Now for any $n \geq 1$, 
\[
d(v, \beta_k(x)) = \sum_{i = 1}^\infty  {2^{-i}\frac{|\beta_k(x)_i - v_i|}{1 + |\beta_k(x)_i - v_i|}} \leq \sum_{i = 1}^n  {2^{-i}\frac{|\beta_k(x)_i - v_i|}{1 + |\beta_k(x)_i - v_i|}} + \sum_{i = n + 1}^\infty {2^{-i}}.
\]
For $n$ sufficiently large, the second sum is $< \epsilon/2$ and then for this $n$, since $(x, 0)$ is transitive for $T_{f_n}$, there exists a $k$ such that the first sum is $< \epsilon/2$ and $d_X(f^k(x), y) < \epsilon$. Then this $k$ does the trick.  This completes the proof. \end{proof}

\section{Additional Remarks}

In this section we will show that given an infranilmanifold $X$ and an Anosov diffeomorphism $T: X\rightarrow X$, there actually exists an inseparable H$\ddot{o}$lder cocycle $f: X\rightarrow \mathbb{R}^\omega$, so the class of functions to which our theorem applies is nonempty. We then make some additional brief remarks on a stronger version of our theorem and on transitivity and stable transitivity of infinite-dimensional Euclidean extensions. 

Recall that an orthant of $\mathbb{R}^n$ is the set of vectors such that each component has some specified sign, e.g. the set of vectors with all of the first $(n - 1)$ components positive and the $n$th component negative.
\begin{lem} Suppose that $p_1, ..., p_{2^n} \in \mathbb{R}^n$ are such that the interior of each orthant of $\mathbb{R}^n$ contains one of the $p_i$. Then for any hyperplane in $\mathbb{R}^n$, there are $p_i$ lying on either side of it. (That is, for each side of the hyperplane, there is a $p_i$ lying on that side of it.) \end{lem}
\begin{proof} Let $0\neq v \in \mathbb{R}^n$. We must show that $\langle v, p_i \rangle > 0$ for some $i$ and $\langle v, p_j \rangle < 0$ for some $j$. For this, just choose $p_i$ and $p_j$ such that for each $k$, the $k$th components of $v$ and $p_i$ have the same sign (where the sign doesn't matter if the $k$th component of $v$ is $0$), and the $k$th components of $v$ and $p_j$ have opposite signs. \end{proof}

\begin{prop} Let $X$ be an infranilmanifold, $T: X\rightarrow X$ an Anosov diffeomorphism. Then there exists an inseparable and hence transitive H$\ddot{o}$lder cocycle $f: X\rightarrow \mathbb{R}^\omega$. \end{prop}

\begin{proof} We will prove that we may actually take the H$\ddot{o}$lder exponent to be $1$. 
We construct the components of $f$ inductively. Suppose that we have already defined an inseparable Lipschitz cocycle 
$f_n:X \rightarrow \mathbb{R}^n$. We then define the cocycle $f_{n + 1}:X\rightarrow \mathbb{R}^{n 
+ 1}$ by $\pi_n \circ f_{n + 1} = f_n$ and letting the $(n + 1)$ component of $f_{n + 1}$ be $g$, 
where $g:X\rightarrow \mathbb{R}$ is to be defined so that $f_{n + 1}$ is inseparable, and $d'_{n + 1}(g(x), g(y)) \leq 2^{-n}d_X(x, y)$ for all $x$, $y\in X$, where $d'_n$ is the restriction of $d$ to the $n$th component of 
$\mathbb{R}^\omega$. It then follows that if we define $f$ by $\pi_n\circ f = f_n$, then $f$ is 
Lipschitz and inseparable (with respect to all finite-dimensional hyperplanes, which are in fact all 
hyperplanes by Proposition 1.2), and so we'll be done. Now we'll want the following
\begin{lem} If $T:X \to X$ Anosov, $f: X \to \R^n$ H$\ddot{o}$lder such that the skew-product $T_f$ is transitive, then the set of weights of $f$ is dense in $\R^n$.
\end{lem}
\begin{proof} 
We use an important basic fact about Anosov diffeomorphisms:
\begin{lem}[Anosov Closing Lemma] Let $T:X \to X$ Anosov. There exist constants $c>0,0<\lambda<1$ such that for any sufficiently small $\epsilon>0$, if $d(T^n(x),x)<\epsilon$, there is a point $p$ with $T^n(p)=p$ and \begin{equation} d(T^i(x),T^i(p)) < c\lambda^{\min(i,n-i)}\epsilon, \hspace{1mm} \forall \hspace{1mm}0\leq i\leq n \end{equation}
\end{lem}
The lemma is proven in, e.g., corollary 6.4.17 in \cite{KH}. Now take a transitive point for $T_f$, say $(x,0)$. Let $w \in \R^n$. For any $\epsilon>0$ there exists $k \in \N$ with $T_f^k(x,0)$ $\epsilon$-close to $(x,w)$. In particular, $d(T^k(x),x) < \epsilon$, so the closing lemma gives a periodic point $p$ with orbit exponentially close to the orbit of $x$. Also note that the sum of $f$ along the first $k$ points of $x$'s orbit is $\epsilon$-close to $w$. Then Lemma 8 in \cite{NP} uses the exponential closeness and the H$\ddot{o}$lder condition to conclude $\sum_{i=0}^{k} ||f(T^i(x)-f(T^i(p)|| < C\epsilon$, $C$ depending only on $T,f$. Thus the weight associated to the point $p$ we get can be made arbitrarily close to $w$. \end{proof}
Since $f_n$ was assumed inseparable, it is transitive by Theorem \ref{npthm}, so it follows from the lemma that we can find $2^{n + 1}$ periodic orbits such that for each orthant of $\mathbb{R}^n$, two of the corresponding $f_n$ sums (where a periodic $f_n$ sum is a sum of $f_n$ along a periodic orbit of $T$) lie
in its interior. For each orthant we define $g$ to have a small $g$ sum along one 
of the corresponding periodic orbits and small negative $g$ sum along the other, where how 
small it has to be is dictated by the condition that 
$d'_{n + 1}(g(x), g(y)) \leq 2^{-n}d_X(x, y)$ for all $x$, $y\in X$. 
We then extend $g$ smoothly to all of $X$ in such a way that this last condition, being Lipschitz with constant $2^-n$, is 
preserved (This can be done, e.g., using partitions of unity.). Then $f_{n + 1}$ contains periodic 
sums in each orthant of $\mathbb{R}^{n + 1}$ and is therefore inseparable by the lemma. This 
completes the proof of the proposition.
\end{proof}

We note that the argument in the proof of Theorem 1.3 may be used to show that given a Hilbert space $H$ and an inseparable H$\ddot{o}$lder cocycle $f: X\rightarrow H$, there exists $x\in X$ such that $(x, 0)$ is transitive for $T_{f_n}$ for each $n$. This can also be easily deduced from our result, since if a map from a compact space is H$\ddot{o}$lder as a map to $H$, then it's H$\ddot{o}$lder as a map to $\mathbb{R}^\omega$, and for both $H$ and $\mathbb{R}^\omega$, the subspace topology for $\mathbb{R}^n$ is just the usual Euclidean topology. Unfortunately, for Hilbert spaces transitivity for each $T_{f_n}$ doesn't immediately imply that $(x, 0)$ is transitive for $T_f$, as it does for $\mathbb{R}^\omega$.

Also a slight variation of our proof gives the stronger conclusion that $T_f$ is topologically weak mixing, meaning the map $T_f \times T_f$ on $(X \times \R^\omega) \times (X \times \R^\omega)$ is transitive. The result is proven for $\R^n$ in \cite{NP}, so we simply take $B'_{knm}$ to now be $\{(x,y) \in X \times X | (T_{f_n} \times T_{f_n})^k((x,0) \times (y,0)) \in U'_{nm}\}$, $U'_{nm}$ now being a basis for $(X \times \R^n) \times (X \times \R^n)$, and the rest of the proof is similar. 

In \cite{NP} it is deduced as a Corollary of Theorem \ref{npthm} that the class of transitive cocycles for a finite-dimensional extension is open in the supremum ($C^0$) norm; that is, every transitive cocycle is actually 
\emph{stably transitive}. No such result holds in our case. In fact, the class of transitive cocycles for an $R^
\omega$ extension has empty interior; that is, no cocycle is stably transitive. Indeed, given a 
cocycle $f: X \rightarrow \mathbb{R}^\omega$, and given $\epsilon > 0$, for $n$ sufficiently large $
\pi_n\circ f$ is an $\epsilon$-perturbation of $f$, i.e. $||f - \pi_n\circ f|| < \epsilon$, where we use the 
supremum norm, because for any $x\in \mathbb{R}^\omega$, $d(x, \pi_n(x)) < \sum_{i = n + 1}^\infty 
{2^{-i}}$. Clearly $\pi_n\circ f$ is not transitive. This observation also shows stable transitivity is impossible in the space of $\ell^p$ cocycles (with its $\ell^p$-metric) for $1\leq p <\infty$ and $c_0$ or $c$ cocyles with the $\ell^\infty$-metric (for $c$, also set the rest of $\pi_n(f(x))$'s coordinates to be the limit of the sequence $f(x)$), as well as the $C^k$ version of these spaces (e.g. if we consider $\ell^2$ cocycles whose partial derivatives up to order $k$ exist and are in $\ell^2$, any cocycle has a nontransitive perturbation whose partial derivatives up to $k$ are arbitrarily $\ell^2$-close). Also, in the H$\ddot{o}$lder metric on cocyles \begin{equation}
d_\alpha(f,g)=\sup_{x \neq y \in X} \frac{d(f(x)-g(x),f(y)-g(y))}{d(x,y)^\alpha}, \end{equation} the cocycle $f$ constructed in the proof above gives an example where stable transitivity fails as well, with cocyles to $\R^\omega$ or $\ell^p$. For instance take the $\ell^1$ metric. The cocycle we constructed satisfies a Lipschitz condition on the $n$-th coordinate with constant $2^{-n}$, so if we take the H$\ddot{o}$lder distance between $f$ and $\pi_n(f)$, we get \begin{equation}
d_1(f,\pi_n(f))=\sup_{x \neq y \in X} \frac{d(f(x)-\pi_nf(x),f(y)-\pi_nf(y))}{d(x,y)} \leq \sup_{x \neq y \in X} \frac{\sum_{i>n} 2^{-i}d(x,y)}{d(x,y)} = 2^{-n} \end{equation}

A similar failure of stable transitivity in the $\R^\omega$ or $\ell^p$ topology on cocyles is true for extensions of compact spaces, not just infranilmanifolds: No continuous cocycle is stably transitive. For let $H$ be, say, a Hilbert space and suppose $f: X
\rightarrow H$ is continuous, and $\epsilon > 0$. Consider the open sets $U_n := \{x\in X : |f(x) - 
\pi_n\circ f(x)| < \epsilon \}$. The $U_n$ form an open cover of $X$, so there is a finite subcover, 
hence $U_N = X$ for some $N$. Then $\pi_N\circ f$ is an $\epsilon$-perturbation of $f$, and 
again, $\pi_N\circ f$ is clearly not transitive.

\bibliographystyle{alpha}

\end{document}